\documentclass[a4paper,fleqn]{cas-sc}

\usepackage[authoryear,longnamesfirst]{natbib}

\def\tsc#1{\csdef{#1}{\textsc{\lowercase{#1}}\xspace}}
\tsc{WGM}
\tsc{QE}
\tsc{EP}
\tsc{PMS}
\tsc{BEC}
\tsc{DE}

\newtheorem{theorem}{Theorem}
 \newtheorem{corollary}[theorem]{Corollary}

\newdefinition{rmk}{Remark}
\newproof{proof}{Proof}

\newcommand{\R}{{\mathbb R}}

\newcommand{\PP}{{\mathbb P}}

\newcommand{\bx}{\mbox{\boldmath{$x$}}}
\newcommand{\bb}{\mbox{\boldmath{$b$}}}

\newcommand{\br}{\mbox{\boldmath{$r$}}}
\newcommand{\bu}{\mbox{\boldmath{$u$}}}
\newcommand{\bv}{\mbox{\boldmath{$v$}}}
\newcommand{\be}{\mbox{\boldmath{$e$}}}
\newcommand{\by}{\mbox{\boldmath{$y$}}}
\newcommand{\sbx}{\mbox{\boldmath{${\scriptstyle x}$}}}
\newcommand{\sby}{\mbox{\boldmath{${\scriptstyle y}$}}}

\newcommand{\bzero}{\mbox{\boldmath{$0$}}}

\begin{document}
\let\WriteBookmarks\relax
\def\floatpagepagefraction{1}
\def\textpagefraction{.001}
\shorttitle{}
\shortauthors{D. Furch\`{i} and L. Reichel}

\title [mode = title]{Iterated Tikhonov regularization of large linear problems}        

\author[1]{D. Furch\`{i}}[orcid=0009-0005-1477-5194]
\ead{d.furchi@uninsubria.it}

\affiliation[1]{organization={Dipartimento di Scienza e Alta Tecnologia, Universit\`a dell'Insubria},
                city={Como},
                postcode={22100},
                country={Italy}}

\author[2]{L. Reichel}[orcid=0000-0003-1729-6816]
\ead{reichel@math.kent.edu}

\affiliation[2]{organization={Department of Mathematical Sciences, Kent State University},
                city={Kent},
                postcode={44242},
                state={Ohio},
                country={USA}}

\begin{abstract}
Many solution methods for linear discrete ill-posed problems with error-contaminated data
(right-hand side) apply Tikhonov regularization to compute a meaningful approximate solution. This solution depends on a 
regularization parameter. It is well known that iterated Tikhonov regularization often 
determines an approximate solution of higher quality than (standard) Tikhonov 
regularization. We consider the situation when an estimate of the norm of the error in the
data is known and would like to apply iterative Tikhonov regularization to determine an 
approximate solution that satisfies the discrepancy principle. This requires a suitable 
choice of a regularization parameter. The standard approach to determine this parameter is
to compute solutions for several values of the regularization parameter and choose a 
computed approximate solution that satisfies the discrepancy principle. This paper 
discusses iterated Tikhonov regularization based on partial Golub-Kahan bidiagonalization
and describes how the regularization parameter can be determined without computing several
approximate solutions by using the connection between Golub-Kahan bidiagonalization and 
Gauss quadrature. This approach reduces the computational effort required to compute a
desired solution.
\end{abstract}


\maketitle


\section{Introduction}\label{sec1}
Many problems in Science and Engineering require the determination of the unknown input
from known, but generally corrupted, output. Such problems are referred to as inverse 
problems. We are concerned with linear inverse problems that are modeled by a 
discretization of a Fredholm integral equation of the first kind with a smooth kernel. The
discretized problem can be written as a linear system of equations
\begin{equation}\label{linsys}
    A{\bx}={\bb}, \qquad A \in {\R}^{m \times n}, \qquad {\bx} \in {\R}^n,\qquad {\bb} \in {\R}^m,
\end{equation}
with a matrix $A$ whose singular values ``cluster'' at the origin, i.e., $A$ has many 
nonvanishing singular values close to the origin; the matrix also may have vanishing
singular values; see, e.g., Chung and Gazzola \cite{CG} and Hansen \cite{Ha2} for a 
discussion on linear systems of equations \eqref{linsys} with a matrix of this kind. A 
nice treatment of inverse problems is provided by Engl et al.  \cite{EHN}. 

The right-hand side vector ${\bb}$ in (\ref{linsys}) represents the available 
error-contaminated output of the linear system. The error, which we denote by 
${\be}\in{\R}^m$ and refer to as ``noise'', may stem from measurement inaccuracies. Due to
the noise in $\bb$, the linear system of equations \eqref{linsys} might not be consistent.

Let $\widehat{\bb}\in{\R}^m$ denote the unknown error-free vector that is associated with the available error-contaminated vector $\bb$, i.e.,
\begin{equation}\label{error}
    {\be}=\widehat{\bb}-{\bb}.
\end{equation}
We assume the linear system of equations 
\begin{equation}\label{linsysexact}
    A{\bx}=\widehat{\bb}
\end{equation}
to be consistent and would like to compute its solution 
$\widehat{\bx}$ of minimal Euclidean 
norm. Thus, our computational task is to determine an approximation of $\widehat{\bx}$ by
computing a suitable approximate solution of the available linear system (\ref{linsys}). 

Straightforward solution of (\ref{linsys}), e.g., solving the least-squares problem
\begin{equation*}
    \min_{\sbx\in\R^n}\|A\bx-\bb\|_2
\end{equation*}
when $m\ge n$, where $\|\cdot\|_2$ denotes the Euclidean norm, generally does not yield a
meaningful approximation of $\widehat{\bx}$ due to the error 
${\be}$ in the data vector ${\bb}$ and the severe ill-conditioning of the matrix $A$. One
therefore often replaces the linear system (\ref{linsys}) with a nearby system, whose 
solution is less sensitive to the error ${\be}$ in ${\bb}$, and uses the solution of the
latter system as an approximation of $\widehat{\bx}$. This replacement is known as 
\emph{regularization}. For ease of discussion, we will assume that a fairly accurate bound
\begin{equation}\label{delta}
    \|\be\|_2\leq\delta
\end{equation} 
is known, but this requirement can be removed; see below.

One of the most popular regularization methods is due to Tikhonov. Tikhonov regularization
replaces (\ref{linsys}) with the least-squares problem 
\begin{equation}\label{tikh}
    \min_{\sbx\in\R^n}\{\|A\bx-\bb\|_2^2+\mu^{-1}\|L\bx\|_2^2\},
\end{equation}
where $\mu>0$ is a regularization parameter and $L\in\R^{s\times n}$ is a regularization
matrix. This matrix often is chosen as the identity matrix or as a finite difference 
matrix with $s\leq n$; see, e.g., \cite{DR,DHNR,Ha2,HNR} for some examples. The value of
$\mu$ determines how sensitive the solution ${\bx}_\mu$ of \eqref{tikh} is to the error 
${\be}$ and how close ${\bx}_\mu$ is to the solution $\widehat{\bx}$ of (\ref{linsysexact}). 
We will below comment on the use of $\mu^{-1}$ in \eqref{tikh} instead of $\mu$. 

Since the error bound \eqref{delta} is known and the system \eqref{linsysexact} is 
consistent, we can determine a suitable value of $\mu>0$ with the discrepancy principle. 
We remark that if no error bound is known, then it may be possible to compute an estimate
of $\delta$ in various ways; see, e.g., \cite{CG,HR1,HR2} for discussions and 
illustrations. Alternatively, one may determine the amount of regularization with other 
techniques such as generalized cross validation; 
see \cite{Ki,KR,RR} for discussions.

It is well-known that \emph{iterated Tikhonov regularization} typically determines more 
accurate approximations of $\widehat{\bx}$ than (standard) Tikhonov regularization 
\eqref{tikh}. Iterated Tikhonov regularization consists in solving a sequence of 
minimization problems
\begin{equation}\label{ittik}
    \bx_{k+1}={\rm arg}\min_{\sbx\in\R^n}\{\|A\bx-\bb\|_2^2+\mu^{-1}\|L(\bx-\bx_k)\|_2^2\},\qquad k=1,2,\ldots,\ell-1,
\end{equation}
where $\bx_1$ is the solution of \eqref{tikh}. The convergence of the solutions 
$\bx_\ell=\bx_{\ell,\mu}$ for $\ell=1,2,\ldots~$, for a suitable value of $\mu$ is 
discussed by Engl et al. \cite{EHN} for the situation when $L=I$, and by Buccini et al.
\cite{BDR} for more general regularization matrices. Computed examples that illustrate the
benefit in accuracy of using iterated Tikhonov regularization \eqref{ittik} instead of 
(standard) Tikhonov regularization \eqref{tikh} can be found in \cite{Bu,BDR,BOR,BPR,HRY}.

The \emph{discrepancy} associated with an approximate solution $\bx_\ell={\bx}_{\ell,\mu}$ 
of \eqref{ittik} is defined as
\begin{equation}\label{rmu}
    {\br}_{\ell,\mu}={\bb}-A{\bx}_{\ell,\mu},\qquad \ell=2,3,\ldots~.
\end{equation}
We also define 
\begin{equation*}
    {\br}_{1,\mu}={\bb}-A{\bx}_{1,\mu},
\end{equation*}
where $\bx_{1,\mu}=\bx_1$ is the solution of the (standard) Tikhonov minimization problem 
\eqref{tikh}. Let $\eta> 1$ be a user-specified parameter. If the noise $\be$ in $\bb$ is 
white Gaussian and the bound $\delta$ in \eqref{delta} is quite accurate, then $\eta$ can 
generally be chosen close to unity. The solutions ${\bx}_{\ell,\mu}$ of (\ref{ittik}) for
$\ell\ge 2$ and ${\bx}_{1,\mu}$ of (\ref{tikh}) are said to satisfy the {\emph discrepancy
principle} if 
\begin{equation}\label{discrprinc}
    \|{\br}_{\ell,\mu}\|_2\approx\eta\delta. 
\end{equation}

The implementations \cite{Bu,BDR,BOR,BPR,HRY} of iterated Tikhonov regularization evaluate
solutions of \eqref{ittik} for a sequence of regularization parameter values $\mu=\mu_k$,
and choose a value such that the solution $\bx_{\ell,\mu}$ satisfies \eqref{discrprinc}. 
For instance, $\mu=\mu_k$ may be determined so that the discrepancy for each solution 
$\bx_k$ of \eqref{ittik}, for $k=1,2,\ldots,\ell$, satisfies \eqref{discrprinc}; see, 
e.g., \cite{HRY}. It is the purpose of the present paper to show how a value of the 
regularization parameter $\mu$ can be determined before the beginning of the iterations 
\eqref{ittik} so that the solution $\bx_{\ell,\mu}=\bx_\ell$ satisfies the discrepancy 
principle. This reduces the computational effort to compute $\bx_\ell$. We achieve this by
first reducing the large-scale problem \eqref{linsys} to a problem of fairly small size 
with the aid of partial Golub-Kahan bidiagonalization of the matrix $A$. Then the 
connection between Golub-Kahan bidiagonalization and Gauss quadrature is used to determine
a suitable value of the regularization parameter.

This paper is organized as follows. Section \ref{sec2} discusses the choice of the
regularization parameter in the iterated Tikhonov method \eqref{ittik} and shows that this
value increases with $\ell$. Partial Golub-Kahan bidiagonalization of the matrix $A$ and 
the connection to Gauss quadrature are reviewed in Section \ref{sec3}. The application of 
partial Golub-Kahan bidiagonalization to speed up the computations required for iterated 
Tikhonov regularization \eqref{ittik} is discussed in Section \ref{sec4}, and Section 
\ref{sec5} contains a few computed examples. Concluding remarks can be found in Section
\ref{sec6}.


\section{Iterated Tikhonov regularization}\label{sec2}

We consider the situation when the regularization matrix $L$ is the identity. An approach
to handling other regularization matrices is commented on below. Let $\bx_{1,\mu}$ solve
the Tikhonov minimization problem \eqref{tikh}. The iterates 
$\bx_{2,\mu},\ldots\bx_{\ell,\mu}$ determined by the iterated Tikhonov method 
\eqref{ittik} satisfy
\begin{equation}\label{ittik2}
    (A^*A+\mu^{-1}I)\bx_{k+1,\mu}=\mu^{-1}\bx_{k,\mu}+A^*\bb,\qquad k=1,2,\ldots,\ell-1.
\end{equation}
It follows that the solution $\bx_{\ell,\mu}$ can be expressed as 
\begin{equation}\label{xmu}
    \bx_{\ell,\mu}=\sum_{j=1}^\ell \frac{1}{\mu^{j-1}}(A^*A+\mu^{-1}I)^{-j}A^*\bb,
\end{equation}
where the superscript $^*$ denotes transposition and $I$ stands for the identity matrix of
suitable order.

\begin{theorem}\label{thm1}
Let $\bx_{\ell,\mu}$ be the iterated Tikhonov approximate solution of \eqref{linsys} 
defined by \eqref{xmu} for some fixed $\mu>0$, and let $\br_{\ell,\mu}$ denote the 
associated residual vector \eqref{rmu}. Then
\begin{equation}\label{rnorm}
    \|\br_{\ell,\mu}\|_2^2=\bb^*(\mu AA^*+I)^{-2\ell}\bb.
\end{equation}
\end{theorem}

\begin{proof}
Let $\mu>0$. We will use the identities
\begin{eqnarray}
\label{eq1}
    (A^*A+\mu^{-1}I)^{-1}A^*&=&A^*(AA^*+\mu^{-1}I)^{-1},\\
\label{eq2}
    AA^*(AA^*+\mu^{-1}I)^{-1}&=&I+(\mu AA^*+I)^{-1}.
\end{eqnarray}
Substituting \eqref{xmu} into \eqref{rmu} and using \eqref{eq1} gives
\begin{equation*}
    \br_{\ell,\mu}=\bb-\sum_{j=1}^\ell \frac{1}{\mu^{j-1}} AA^*(AA^*+\mu^{-1})^{-j}\bb.
\end{equation*}
Application of \eqref{eq2} now yields
\begin{equation*}
    \br_{\ell,\mu}=(\mu AA^*+I)^{-\ell}\bb
\end{equation*}
and the theorem follows. $\Box$
\end{proof}

The discrepancy principle \eqref{discrprinc} prescribes that for a given $\ell\ge 1$, the
regularization parameter $\mu>0$ be chosen so that
\begin{equation}\label{disc2}
    \bb^*(\mu AA^*+I)^{-2\ell}\bb=\eta^2\delta^2.
\end{equation}

We use the singular value decomposition (SVD)
\begin{equation}\label{svd}
    A=U\Sigma V^*,
\end{equation}
in our analysis below. Here the matrices $U\in\R^{m\times m}$ and $V\in\R^{n\times n}$ are
orthogonal and
\begin{equation*}
    \Sigma={\rm diag}[\sigma_1,\sigma_2,\ldots,\sigma_{\min\{m,n\}}]\in\R^{m\times n}.
\end{equation*}
The singular values $\sigma_j$ are nonnegative and ordered according to 
$\sigma_1\ge\sigma_2\ge\ldots\ge\sigma_{\min\{m,n\}}\ge 0$. We remark that the numerical 
method proposed in this paper does not require the computation of the SVD.

\begin{corollary}\label{cor1}
Using the notation of Theorem \ref{thm1}, let $0<\mu<\infty$ and assume that $A^*\bb\ne\bzero$. Then
\begin{equation*}
    \|\br_{\ell+1,\mu}\|_2<\|\br_{\ell,\mu}\|_2,\qquad \ell=1,2,\ldots~.
\end{equation*}
Further, let $\bb_0$ denote the orthogonal projections of $\bb$ onto the null space of 
$A^*$. Then
\begin{equation}\label{limit}
    \lim_{\ell\rightarrow\infty}\|\br_{\ell,\mu}\|_2=\|\bb_0\|_2.
\end{equation}
In particular, in the common situation when $n=m$ and the matrix $A\in\R^{m\times n}$ is of full 
rank, we have $\lim_{\ell\rightarrow\infty}\|\br_{\ell,\mu}\|_2=0$. Finally, it follows by continuity that
\begin{equation}\label{lim0}
    \lim_{\mu\searrow 0}\|\br_{\ell,\mu}\|_2=\|\bb\|_2,\quad \ell=1,2,\ldots~.
\end{equation}
\end{corollary}
\begin{proof} 
Let 
\begin{equation*}
\widetilde{\bb}=[\widetilde{b}_1,\widetilde{b}_2,\ldots,\widetilde{b}_{m}]^*=U^*\bb,
\end{equation*}
where $U$ is determined by the SVD \eqref{svd}. Then $A^*\bb\ne\bzero$ is equivalent to 
the existence of an index $1\le i\le {m}$, such that $\sigma_i\widetilde{b}_i\ne 0$, where
$\sigma_i$ is the $i$th singular value of $A$. Using the SVD, we obtain
\begin{eqnarray}
\nonumber
\|\br_{\ell,\mu}\|_2^2&=&\bb^*(\mu AA^*+I)^{-2\ell}\bb=
\widetilde{\bb}^*(\mu\Sigma\Sigma^*+I)^{-2\ell}\widetilde{\bb}\\
\label{bell}
&=&\sum_{j=1}^m\frac{\widetilde{b}_j^2}{(\mu\sigma_j^2+1)^{2\ell}}\\
\nonumber
&>&\sum_{j=1}^m\frac{\widetilde{b}_j^2}{(\mu\sigma_j^2+1)^{2\ell+2}}=
\|\br_{\ell+1,\mu}\|_2^2,
\end{eqnarray} 
where the last inequality follows from the fact that $\mu\sigma_j^2+1\ge 1$ for all $j$
and $\mu\sigma_i^2+1>1$.

We turn to the limit \eqref{limit}. It follows from \eqref{bell} that
\begin{equation*}
\lim_{\ell\rightarrow\infty} \|\br_{\ell,\mu}\|_2^2 = 
\lim_{\ell\rightarrow\infty} 
\sum_{j=1}^m\frac{\widetilde{b}_j^2}{(\mu\sigma_j^2+1)^{2\ell}} =
\sum_{j=1\atop \sigma_j=0}^m \widetilde{b}_j^2 = 
\bb_0^*\bb_0.
\end{equation*}
Finally, the limit \eqref{lim0} is an immediate consequence of the representation 
\eqref{rnorm}. This completes the proof. $\Box$
\end{proof}

\begin{theorem}\label{thm2}
Assume that $A^*\bb\ne\bzero$. Then the function $\mu\to\|\br_{\ell,\mu}\|_2^2$ is 
strictly decreasing and strictly convex for $\mu>0$. 
\end{theorem}

\begin{proof}
The first and second derivatives are given by
\begin{eqnarray*}
\frac{\partial}{\partial\mu}\|\br_{\ell,\mu}\|_2^2&=& 
-2\ell\,\bb^* AA^*(\mu AA^*+I)^{-2\ell-1}\bb<0,\\
\frac{\partial^2}{\partial\mu^2}\|\br_{\ell,\mu}\|_2^2&=&
(2\ell+1)2\ell\,\bb^* AA^*(\mu AA^*+I)^{-2\ell-2}AA^*\bb>0.
\end{eqnarray*}
The fact that the right-hand side inequalities are strict follows from $A^*\bb\ne\bzero$. $\Box$
\end{proof}

We would like to determine the regularization parameter $\mu$ so that that 
$\bx_{\ell,\mu}$ satisfies the discrepancy principle, i.e., so that \eqref{disc2} holds. Moreover, we are interested in how $\mu=\mu_\ell$ depends on the number of iterations $\ell$ in 
\eqref{ittik}. 

\begin{corollary}\label{cor2}
Assume that $\delta>0$ and $\eta> 1$, and let $\mu=\mu_\ell>0$ satisfy \eqref{disc2} for
some $\ell\ge 1$. Then $\mu_\ell$ is a decreasing function of $\ell$.
\end{corollary}

\begin{proof}
It follows from Corollary \ref{cor1} that
\begin{equation*}
\|\br_{\ell+1,\mu_{\ell}}\|_2<\|\br_{\ell,\mu_{\ell}}\|_2=\eta\delta.
\end{equation*}
Let $\mu=\mu_{\ell+1}$ be such that $\br_{\ell+1,\mu_{\ell+1}}$ satisfies \eqref{disc2}.
Then, by Theorem \ref{thm2}, we have $\mu_{\ell+1}\le\mu_{\ell}$. We conclude that 
the values of $\mu=\mu_{\ell}$ that satisfy \eqref{disc2} decrease as $\ell$ increases. $\Box$
\end{proof}

The amount of regularization in \eqref{ittik} is determined by $1/\mu$. By Corollary 
\ref{cor2}, this parameter increases with $\ell$. 


\section{Partial Golub-Kahan bidiagonalization and Gauss quadrature}\label{sec3}
This section first discusses the application of the partial Golub-Kahan bidiagonalization
method to the solution of the iterated Tikhonov minimization problem 
\eqref{ittik} when the regularization matrix $L$ is the identity. A modification that 
allows a more general regularization matrix is commented on at the end of this section. 
The close connection between Golub-Kahan bidiagonalization and Gauss-type quadrature rules
allows the computation of inexpensive bounds that indicate how accurately the computed
solution satisfies the discrepancy principle \eqref{discrprinc}.

A nice discussion on the application of partial Golub-Kahan bidiagonalization to the solution of
linear discrete ill-posed problems is provided by Bj\"orck \cite{Bj1}. The application of Gauss-type quadrature rules to 
estimate how accurately the computed solution satisfies the discrepancy principle when 
Tikhonov regularization \eqref{tikh} is used is described in \cite{CR}; see also 
\cite{RSS}. A nice overview of the relation between Golub-Kahan bidiagonalization and 
Gauss quadrature is provided by Golub and Meurant \cite{GM}; see also Calvetti et al. 
\cite{CGR}. We will apply this approach to compute upper and lower bounds for
the left-hand side of \eqref{discrprinc}. To this end, introduce the functions
\begin{equation}\label{phi}
\phi_\ell(\mu)=\|\br_{\ell,\mu}\|_2^2,\qquad \ell=1,2,\ldots~,
\end{equation}
where ${\br}_{\ell,\mu}$ is defined by \eqref{rmu}. Several properties of these functions
have been established in Section \ref{sec2}. Here we are concerned with their inexpensive
approximation with the aid of Gauss quadrature as follows. Express $\phi_\ell(\mu)$ with
aid of a Stieltjes integral. The SVD of $A$ \eqref{svd} yields the spectral factorization
\begin{equation}\label{spectfact}
AA^*=U\Lambda U^*, 
\end{equation}
where
\begin{equation*}
\Lambda={\rm diag}[\lambda_1,\lambda_2,\ldots,\lambda_m]=\Sigma\Sigma^*\in\R^{m\times m}.
\end{equation*}
Thus, 
\begin{equation*}
\lambda_j=\left\{\begin{array}{cc} \sigma_j^2,~~ & 1\le j\le \min\lbrace n,m\rbrace, \\
                                   0,~~ & \min\lbrace n,m\rbrace<j\le m. \end{array}\right.
\end{equation*}
Using the notation in the proof of Corollary \ref{cor1} and the above definition of the
eigenvalues $\lambda_j$ of $AA^*$, we obtain 
\begin{equation*}
\phi_\ell(\mu)=\bb^*U(\mu\Lambda+I)^{-2\ell}U^*\bb =
\sum_{j=1}^m\frac{\widetilde{b}_j^2}{(\mu\lambda_j+1)^{2\ell}},\qquad \ell=1,2,\ldots~.
\end{equation*}
Following \cite{CGR,CR,GM}, we interpret the sum in the right-hand side as a Stieltjes 
integral associated with a nondecreasing piece-wise constant distribution function 
$\omega$ that has jumps at the eigenvalues $\lambda_j$. This allows us to express the
functions $\phi_\ell$ as
\begin{equation}\label{intgral}
\phi_\ell(\mu)=\int_0^\infty \frac{1}{(\mu t + 1)^{2\ell}}d\omega(t),\qquad 
\ell=1,2,\ldots~.
\end{equation}

\begin{corollary}
Assume that $A^*\bb\ne\bzero$ and let the function $\mu\to\phi_\ell(\mu)$ for some integer
$\ell\geq 1$ be defined by \eqref{phi} for $\mu>0$. Then the equation
equation
\begin{equation}\label{phieq}
\phi_\ell(\mu)=\tau
\end{equation}
has a unique solution $\mu$, such that $0<\mu<\infty$, for any $\tau$ that 
satisfies $\|{\bb}_0\|^2<\tau<\|{\bb}\|^2$, where ${\bb}_0$ denotes the
orthogonal projection of ${\bb}$ onto the null space of $A^*$.
\end{corollary}

\begin{proof}
By Theorem \ref{thm2}, the function $\mu\to\phi_\ell(\mu)$ is decreasing and convex for
$\mu\geq 0$, where $\phi_\ell(0)=\|\bb\|_2^2$ is defined by continuity. Moreover, 
\begin{equation*}
\lim_{\mu\rightarrow\infty}\phi_\ell(\mu)=\|\bb_0\|_2^2.
\end{equation*}
This shows the corollary. $\Box$
\end{proof}

Since $\phi_\ell(\mu)$ is a convex and decreasing function of $\mu$, equation \eqref{phieq} can be solved by Newton's method without safe-guarding; see \cite{CR} for details.
We note that the function $\nu\to\phi_\ell(1/\nu)$ might not be convex. The solution of the equation $\phi_\ell(1/\nu)=\tau$ by Newton's method therefore generally requires that Newton's method be safe-guarded. This is the reason why we use the regularization parameter $\mu$ in (\ref{ittik}) instead of the parameter $\nu=1/\mu$. Discussions on zero-finders can be found in \cite{BPR,CR,RS}.

Our solution method for \eqref{discrprinc} is based on partial Golub-Kahan 
bidiagonalization of the matrix $A$ with initial vector ${\bb}$. Application of 
$1\le q\ll\min\{m,n\}$ bidiagonalization steps gives the decompositions
\begin{equation}\label{lbd}
AV_q=U_{q+1}\bar{C}_q, \qquad A^*U_q = V_qC_q^*, 
\end{equation}
where the matrices $U_{q+1}=[\bu_1,\ldots,\bu_{q+1}]^*\in{\R}^{m\times(q+1)}$ and 
$V_q=[\bv_1,\ldots,\bv_q]^*\in{\R}^{n\times q}$ have orthonormal columns, 
$\bu_1=\bb/\|\bb\|_2$, and the lower bidiagonal matrix
\begin{equation}\label{Cbarmat}
\bar{C}_q=\left[ \begin{array} {ccccc}
\rho_1 & & & & \mbox{\Large 0} \\
\sigma_2& \rho_2 & & &  \\
 & \ddots & \ddots &  &  \\
 & & \sigma_{q-1} & \rho_{q-1} & \\
 & & & \sigma_q & \rho_q \\
\mbox{\Large 0} & & & & \sigma_{q+1} \\
\end{array}
\right]\in{\R}^{(q+1)\times q}
\end{equation}
has positive entries $\rho_j$ and $\sigma_j$. Moreover, $C_q\in\R^{q\times q}$ is made up 
of the first $q$ rows of $\bar{C}_q$ and $U_q\in{\R}^{m\times q}$ consists of the $q$ 
first columns of $U_{q+1}$. We assume that the number of bidiagonalization steps, $q$, is 
chosen small enough so that the decompositions (\ref{lbd}) with the stated properties 
exist. This is the generic situation. In the rare event that $\sigma_j>0$ for 
$1\leq j\leq q$, and $\sigma_{q+1}=0$, the computations simplify. We will not dwell on 
this unusual situation. To secure that the columns of the
matrices $U_{q+1}$ and $V_q$ are numerically orthogonal, we carry out partial Golub-Kahan
bidiagonalization with reorthogonalization. Since the number of steps, $q$, typically is 
quite small, this does not increase the computational effort much. 

\begin{theorem}\label{thm2b}
Introduce the functions
\begin{eqnarray}
\label{phiell}
\phi_{q,\ell}(\mu)&=&\|{\bb}\|_2^2 {\be}_1^*(\mu C_q C_q^*+I)^{-2\ell}{\be}_1, \\
\label{phibarell}
\bar{\phi}_{q+1,\ell}(\mu)&=&\|{\bb}\|_2^2 
{\be}_1^*(\mu\bar{C}_q\bar{C}_q^*+I)^{-2\ell}{\be}_1,
\end{eqnarray}
where $\be_1=[1,0,\ldots,0]^*$ denotes the first axis vector of appropriate dimension. The
function \eqref{phiell} is a $q$-node Gauss quadrature rule and the function 
\eqref{phibarell} is an $(q+1)$-node Gauss-Radau quadrature rule with a fixed node at the
origin, for the approximation of the integral \eqref{intgral}. Moreover,
\begin{equation}\label{phiellineq}
\phi_{q,\ell}(\mu)<\phi_{\ell}(\mu)<\bar{\phi}_{q+1,\ell}(\mu), \qquad \mu>0.
\end{equation}
\end{theorem}

\begin{proof}
The decompositions \eqref{lbd} gives
\begin{equation}\label{lanczos}
AA^*U_q=U_q C_q C_q^*+\rho_q\sigma_{q+1}\bu_{q+1}\be_1^*,
\end{equation}
where we note that the matrix $T_q=C_qC_q^*$ is symmetric and tridiagonal. The expression
\eqref{lanczos} also can be determined by applying $q$ steps of the symmetric Lanczos 
process to the matrix $AA^*$ with initial vector $\bb$. The fact that $T_q$ is 
tridiagonal shows that the columns $\bu_j$ of $U_q$ satisfy a three-term recurrence with 
the nontrivial entries of $T_q$ being the recursion coefficients. 

Let $f$ be an analytic function defined on the convex hull of the spectrum of $AA^*$.
Then the quantity $\|\bb\|_2^2\be_1^*f(T_q)\be_1$ can be interpreted as a $q$-node Gauss 
quadrature rule for the approximation of $\bb^*f(AA^*)\bb$. The underlying measure is 
$d\omega$ in \eqref{intgral} and the nontrivial entries of the matrix $T_q$ are the 
recurrence coefficients for orthonormal polynomials associated with this measure. Thus,
the orthonormal polynomials satisfy a recurrence relation with the same recursion
coefficients as the columns of the matrix $U_q$. This provides the connection between the
symmetric Lanczos process applied to $AA^*$ and Gauss quadrature. 

Let $\PP_k$ denote the
set of all polynomials of degree less than or equal to $k$. The Gauss quadrature rule
$\|\bb\|_2^2\be_1^*f(T_q)\be_1$ is exact for all polynomials in $\PP_{2q-1}$, i.e.,
\begin{equation*}
\int_0^\infty f(t)d\omega(t)=\|\bb\|_2^2\be_1^*f(T_q)\be_1,\qquad \forall f\in\PP_{2q-1}.
\end{equation*}
Details of the relation between the symmetric Lanczos process and Gauss quadrature are 
provided by Golub and Meurant \cite{GM}; discussions relevant for the application of the 
present paper can be found in \cite{CGR,CR}.

Introduce the tridiagonal matrices 
\begin{equation*}
T_{q+1}=C_{q+1}C_{q+1}^*\in\R^{(q+1)\times(q+1)},\qquad
\bar{T}_{q+1}=\bar{C}_q\bar{C_q}^*\in\R^{(q+1)\times(q+1)}.
\end{equation*}
The matrix $\bar{T}_{q+1}$ differs from $T_{q+1}$ only in the last diagonal entry. By the 
definition of $\bar{T}_{q+1}$, this matrix has one vanishing eigenvalue. It follows that 
$\|\bb\|_2^2\be_1^*f(\bar{T}_{q+1})\be_1$ is a $(q+1)$-node Gauss-Radau quadrature rule 
with a specified node at the origin. This quadrature rule satisfies
\begin{equation*}
\int_0^\infty f(t)d\omega(t)=\|\bb\|_2^2\be_1^*f(\bar{T}_{q+1})\be_1,\qquad \forall 
f\in\PP_{2q};
\end{equation*}
see \cite{CGR,CR} for proofs and discussions.

We turn to the inequalities \eqref{phiellineq} and first consider the remainder for the
Gauss rule $\|\bb\|_2^2\be_1^*\psi(T_q)\be_1$, where
\begin{equation}\label{psi}
\psi(t)=(\mu t +1)^{-2\ell}.
\end{equation}
Since for any positive fixed value of $\mu$, the derivatives of $\psi$ with respect to $t$
of even order are strictly positive, the remainder is positive. This yields the left-hand
side inequality in \eqref{phiellineq}. The derivatives of $\psi$ of odd order are strictly
negative. The remainder for Gauss-Radau quadrature therefore is negative. This gives the
right-hand side inequality in \eqref{phiellineq}. $\Box$
\end{proof}

We would like to determine $\mu$ so that \eqref{discrprinc} holds. Therefore, we seek
to find a possibly small value of $q$ and values of $\mu$ so that
\begin{equation}\label{2ineq}
\delta^2\leq\phi_{q,\ell}(\mu),\qquad \bar{\phi}_{q+1,\ell}(\mu)\leq \eta^2\delta^2.
\end{equation}
Then it follows from (\ref{phiellineq}) and (\ref{2ineq}) that $\delta^2<\phi_\ell(\mu)<\eta^2\delta^2$. In view of \eqref{discrprinc} and \eqref{phi}, the computed solution 
$\bx_{\ell,\mu}$ satisfies the discrepancy principle.

We conclude this section with a discussion of some properties of the functions
$\phi_{q,\ell}$ and $\bar{\phi}_{q+1,\ell}$. 

\begin{theorem}\label{thm2a}
Assume that $\rho_q\sigma_{q+1}>0$. Then 
\begin{eqnarray}
\label{philim}
\lim_{\mu\rightarrow\infty}\phi_{q,\ell}(\mu)&=&0, \\
\label{barphilim}
\lim_{\mu\rightarrow\infty}\bar{\phi}_{q+1,\ell}(\mu)&=&\omega_0, 
\end{eqnarray}
where $\omega_0>0$ is the weight of the $(q+1)$-point Gauss-Radau quadrature rule 
associated with the node at the origin.
\end{theorem}

\begin{proof}
This result is shown in \cite{CR} for the case when $\ell=1$. The result carries over to
the situation when $\ell>1$. The limit \eqref{philim} follows from the fact that under the
assumption of the theorem, the matrix $C_qC_q^*$ is symmetric positive definite. The 
matrix $\bar{C}_q\bar{C}_q^*$ is symmetric positive semidefinite with one eigenvalue at 
the origin. This yields the limit \eqref{barphilim}. $\Box$
\end{proof}

\begin{theorem}\label{thm2c}
The function $\mu\to\phi_{q,\ell}(\mu)$ defined by (\ref{phiell}) is strictly decreasing
and convex. Assume that $\rho_q\sigma_{q+1}>0$. Then the equation
\begin{equation}\label{phielleq}
\phi_{q,\ell}(\mu)=\tau
\end{equation}
has a unique solution $\mu$ such that $0<\mu<\infty$ for any $\tau$ with 
$0<\tau<\|{\bb}\|_2^2$.
\end{theorem}

\begin{proof}
A proof for the case when $\ell=1$ is provided in \cite{CR}. This proof carries over to 
the present situation. $\Box$
\end{proof}

\begin{theorem}\label{thm3}
Assume that the partial Golub-Kahan decompositions (\ref{lbd}) exists with 
$\rho_\ell\sigma_{\ell+1}>0$. Then the functions $\phi_{q,\ell}$ and $\phi_{q+1,\ell}$ are
well defined and satisfy
\begin{equation}\label{phiellmon}
\phi_{q,\ell}(\mu)<\phi_{q+1,\ell}(\mu), \qquad \mu\geq 0.
\end{equation}
\end{theorem}

\begin{proof}
A proof can be found in \cite{LRW}. A related result is show in \cite{CR}. $\Box$
\end{proof}


\section{Computation of approximate solutions}\label{sec4}
We describe how to evaluate an approximate solution of equation (\ref{ittik}) and start
with the computation of an approximate solution of \eqref{tikh}. Assume that the partial 
Golub-Kahan decomposition (\ref{lbd}) of $A$ and a suitable value $\mu>0$ of the 
regularization parameter are available. We first seek to determine an approximate solution
of \eqref{tikh} of the form
\begin{equation}\label{xell}
{\bx}_{\mu}=V_q{\by}_{\mu},
\end{equation}
where ${\by}_{\mu}\in{\R}^q$ solves the Galerkin equation 
\begin{equation}\label{gal1}
V_q^*(A^*A+\mu^{-1} I)V_q{\by}_{\mu}=V_q^*A^*{\bb}.
\end{equation}
Then \eqref{xell} is an approximate solution of (\ref{tikh}). Using the right-hand side 
decomposition in (\ref{lbd}) and the fact that $U_{q+1}^*\bb=\be_1\|\bb\|_2$, equation 
(\ref{gal1}) simplifies to
\begin{equation}\label{galsimp}
(\bar{C}_q^*\bar{C}_q+{\mu}^{-1} I){\by}_{\mu}=\bar{C}_q^*{\be}_1\|\bb\|_2.
\end{equation}
These are the normal equations associated with the least-squares problem
\begin{equation}\label{ysol}
\min_{{\sby}_{\mu}\in {\R}^q}\left\|
 \left[ \begin{array} {c} \bar{C}_q \\ \mu^{-1/2}I \end{array} \right]
{\by}_{\mu} - {\be}_1\|\bb\|_2  \right\|_2.
\end{equation}
We determine the solution ${\by}_{\mu}$ of (\ref{galsimp}) by solving (\ref{ysol}), 
because the condition number of the system matrix in \eqref{ysol} is the square root of 
the condition number of the system matrix of \eqref{galsimp}. The least-squares problem
\eqref{ysol} can be computed in only ${\mathcal O}(q)$ arithmetic floating point 
operations by using Givens rotations. 

\begin{theorem}
Let ${\by}_{\mu}$ solve (\ref{galsimp}), and let $\bx_{\mu}$ be given by \eqref{xell}. 
Then 
\begin{equation}\label{discxell}
\|{\bb}-A{\bx}_{\mu}\|_2=\bar{\phi}_{q+1,1}^{1/2}(\mu),
\end{equation}
where $\bar{\phi}_{q+1,1}$ is defined by \eqref{phibarell}.
\end{theorem}

\begin{proof}
We obtain from \eqref{xell} and \eqref{galsimp} that
\begin{eqnarray*}
\bb-A\bx_{\mu}&=&
\bb-AV_q(\bar{C}_q^*\bar{C}_q+\mu^{-1}I)^{-1}\bar{C}_q^*\be_1\|\bb\|_2\\
&=& \bb-U_{q+1}\bar{C}_q(\bar{C}_q^*\bar{C}_q+\mu^{-1}I)^{-1}\bar{C}_q^*\be_1\|\bb\|_2\\
&=& U_{q+1}(I-\bar{C}_q(\bar{C}_q^*\bar{C}_q+\mu^{-1}I)^{-1}\bar{C}_q^*)\be_1
\|\bb\|_2.
\end{eqnarray*}
Using the identities \eqref{eq1} and \eqref{eq2} with $A$ replaced by $\bar{C}_q$ gives
\begin{equation*}
\bb-A\bx_{\mu}=-U_{q+1}(\mu\bar{C}\bar{C}^*+I)^{-1}\be_1\|\bb\|_2.
\end{equation*}
This shows \eqref{discxell}. $\Box$
\end{proof}

We turn to the iterated Tikhonov method. Let $\by_\mu$ solve the Galerkin equation 
\eqref{gal1}, and define $\bx_{1,\mu}=V_q\by_\mu$. The Galerkin equations associated with 
the iterated Tikhonov method \eqref{ittik2} with $\bx_{k,\mu}=V_q\by_{k,\mu}$ are given by
\begin{equation}\label{ittikgal}
V_q^*(A^*A+\mu^{-1}I)V_q\by_{k+1,\mu}=\mu^{-1}\by_{k,\mu}+V_q^*A^*\bb,\qquad 
k=1,2,\ldots,\ell-1.
\end{equation}
Simplifications analogous to those that gave \eqref{galsimp} yield
\begin{equation*}
(\bar{C}_q^*\bar{C}_q+\mu^{-1}I)\by_{k+1,\mu}=
\mu^{-1}\by_{k,\mu}+\bar{C}_q^*\be_1\|\bb\|_2, \qquad k=1,2,\ldots,\ell-1,
\end{equation*}
and we obtain similarly to the expression \eqref{xmu},
\begin{equation}\label{ymu}
\by_{\ell,\mu}=\sum_{j=1}^\ell \frac{1}{\mu^{j-1}}(\bar{C}_q^*\bar{C}_q+\mu^{-1}I)^{-j}
\bar{C}_q^*\be_1\|\bb\|.
\end{equation}
Consider the residual vector
\begin{eqnarray*}
\br_{\ell,\mu}&=&\bb-AV_q\by_{\ell,\mu}=\bb-U_{q+1}\bar{C}_q\by_{\ell,\mu}\\
&=&U_{q+1}\left(I-\sum_{j=1}^\ell \frac{1}{\mu^{j-1}}
\bar{C}_q\left(\bar{C}_q^*\bar{C}_q+\mu^{-1}I\right)^{-j} \bar{C}_q^*\right)\be_1\|\bb\|.
\end{eqnarray*}
Using the identities \eqref{eq1} and \eqref{eq2} with $A$ replaced by $\bar{C}_q$ allows
us to simplify the above expression to
\begin{equation}\label{rellmu}
\br_{\ell,\mu}=-U_{q+1}(\mu \bar{C}\bar{C}^*+I)^{-\ell}\be_1\|\bb\|_2.
\end{equation}
We have shown the following result.

\begin{theorem}\label{thm5}
Let ${\by}_{\ell,\mu}$ be defined by (\ref{ymu}) and let 
$\bx_{\ell,\mu}=V_q\by_{\ell,\mu}$. Then 
\begin{equation}\label{discxell2}
\|{\bb}-A{\bx}_{\ell,\mu}\|_2=\bar{\phi}_{q+1,\ell}^{1/2}(\mu).
\end{equation}
\end{theorem}

\begin{proof}
We obtain from \eqref{rellmu} that
\begin{equation*}
\|\br_{\ell,\mu}\|_2^2=\|\bb\|_2^2 \be_1^*(\mu \bar{C}_q\bar{C}_q^{\ast}+I)^{-2\ell}\be_1.
\end{equation*}
Comparison with \eqref{phibarell} shows the theorem. $\Box$
\end{proof}

Theorems \ref{thm2b} and \ref{thm5} show that the Galerkin solution of the iterated
Tikhonov equations gives a residual vector that is of larger norm than if the computations
were carried out with the full matrix $A$. Moreover, Theorem \ref{thm2b} shows that a 
lower bound for the residual norm of the exact solution of the iterated Tikhonov method 
can be computed by evaluating $\phi_{q,\ell}^{1/2}(\mu)$. Increasing the number of
Golub-Kahan bidiagonalization steps leads to tighter bounds for the unknown residual of 
the problem with the full matrix $A$. This suggests the following method for iterated
Tikhonov regularization:

\begin{enumerate}
\item Determine an initial value of $\mu>0$, the number of Golub-Kahan 
bidiagonalization steps $q\geq 1$ and the number of iterations $\ell$ of our method.
\item Determine $\mu>0$ so that
$\bar{\phi}_{q+1,\ell}^{1/2}(\mu)=\eta\delta$ if possible. If this is not possible, then
increase $q$ by one. A zero-finder has to be used. There are many different possible choices of zero-finders such as a cubically convergent
zero-finder described in \cite{RS}. In the computed examples, we use a bisection method.
\item
Compute the lower bound $\phi_{q,\ell}^{1/2}(\mu)$. If 
$\bar{\phi}_{q+1,\ell}^{1/2}(\mu)-\phi_{q,\ell}^{1/2}(\mu)\leq \eta\delta/100$, then
accept the available Galerkin solution as an approximation of the solution of
\eqref{ittik}. If this inequality does not hold, then increase $q$ and compute a new
value of $\mu$ as described in Step 2.
\end{enumerate}

The factor $1/100$ in the bound of Step 3 is, of course, somewhat arbitrary and can be
changed.

We finally comment on the use of a regularization matrix $L\ne I$ in \eqref{ittik}.
Both Bj\"orck \cite{Bj1} and Eld\'en \cite{El} describe how Tikhonov minimization
problems \eqref{tikh} with $L\ne I$ can be transformed to an equivalent problem with
$L=I$. This transformation also can be used for the iterated Tikhonov method. The use of this transformation  is attractive
when $L$ has exploitable structure, such as being banded with small bandwidth. This
technique also can be applied when $L$ is an orthogonal projector; see \cite{MRS} for
illustrations of the latter.


\section{Computed examples}\label{sec5}
All computations were carried out using MATLAB with about $15$ significant decimal digits.
To ensure reproducibility, we set $\sf{seed}=11$ for the random number generator. The 
random number generator is used to determine the error vector $\be$ in the data, where $\be$ 
models white Gaussian noise. In the computed examples, we set $\eta=1$ in \eqref{discrprinc}, 
\eqref{disc2}, and onwards. While the theory that supports the discrepancy 
principle in infinite-dimensional Hilbert space prescribes that $\eta>1$, the method presented
performs well for $\eta=1$ for finite-dimensional problems with white Gaussian noise $\be$.
We remark that the inequality
\[
\bar{\phi}_{q+1,\ell}^{1/2}(\mu)-\phi_{q,\ell}^{1/2}(\mu)\leq \eta\delta/100
\]
holds for all of our examples.

We compare the results achieved with the method described in this paper with those
determined with an iterated Golub-Kahan-Tikhonov method described in \cite{bdfrGK}. 
The latter method is derived and convergence is shown in (infinite-dimensional) Hilbert space; the method is an extension of the approach by Neubauer \cite{Ne} to iterated Tikhonov regularization. The values of the regularization parameters of the method of the present paper and the method in \cite{bdfrGK} differ. In particular, we use the improved parameter choice strategy discussed in \cite[Proposition 7]{bdfrGK}, which we now recall: Consider the singular value decomposition 
\begin{equation*}
    \bar{C}_{q}=W_{q+1}\bar{\Sigma}_q S_{q}^{\ast},
\end{equation*}
where $W_{q+1}\in\R^{(q+1)\times(q+1)}$ and $S_q\in\R^{q\times q}$ are orthogonal matrices, and 
$\bar{\Sigma}_q\in\R^{(q+1)\times q}$ is a diagonal matrix whose nontrivial entries are the singular values of 
$\bar{C}_q$. Let 
$\tilde{\mu}$ be the unique solution of the equation
\begin{equation}\label{condgk2}
\bb_q^{\ast}(\mu\Sigma_{q}
\Sigma_{q}^{\ast}+I_{q+1})^{-2\ell-1}\bb_q=\delta^2,
\end{equation}
where
\begin{equation*}
    \bb_q=I_{q,q+1}W_{q+1}^{\ast}U_{q+1}^{\ast}\bb\qquad\text{and}\qquad I_{q,q+1}=\begin{bmatrix}
  I_{q} & 0\\
  0 & 0
  \end{bmatrix}\in\R^{(q+1)\times(q+1)}.
\end{equation*}

This choice of $\tilde{\mu}$ is in \cite{bdfrGK} shown to yield the smallest relative error for the parameter in the interval $[\tilde{\mu},+\infty)$ under the additional hypothesis that
\begin{equation*}
    \left(U_{q+1}W_{q+1}I_{q,q+1}W_{q+1}^{\ast}U_{q+1}^{\ast}A-AV_qV_q^{\ast}\right)\widehat{\bx}=0.
\end{equation*}
In the following examples this condition is not satisfied. Nonetheless, we can apply this choice of the regularization parameter and compare it with the choice of the regularization parameter proposed in the present paper. 

We also compare the method of the present paper to the {\sf lsqr} method 
provided by MATLAB. The latter is an iterative method that determines a partial
Golub--Kahan bidiagonalization of $A$ with initial vector $\bb$; see \cite{PS}. 
Regularization is 
achieved by truncated iteration, i.e., the iterations are terminated as soon as an 
iterate has been determined that is associated with a residual error of norm at most 
$\eta\delta$; see \cite{Ha2,RSS} for details of this regularizaation method.

\vspace{0.5cm}


Example 5.1. Consider the Fredholm integral equation of the first kind,
\begin{equation}\label{shaw}
\int_{-\pi/2}^{\pi/2} \kappa(\sigma,\tau)x(\sigma)d\sigma=b(\tau),\qquad
-\frac{\pi}{2}\le \tau\le \frac{\pi}{2},
\end{equation}
where 
\begin{equation*}
   \kappa(\sigma,\tau):=(\cos(\sigma)+\cos(\tau))^2\left(\frac{\sin(\xi)}{\xi}\right)^2,
\qquad \xi:=\pi(\sin(\sigma)+\sin(\tau)),
\end{equation*}
and the right-hand side $b(\tau)$ is chosen so that the solution $x(\sigma)$ 
is a sum of two Gaussian functions. This integral equation is discussed by 
Shaw \cite{Sh}. We use the code {\sf shaw} from the MATLAB package 
Regularization Tools by Hansen \cite{Ha1} to discretize (\ref{shaw}) by a 
quadrature rule with $100$ nodes. This yields the matrix 
$A\in{\R}^{100\times 100}$ and right-hand side 
$\widehat{\bb}\in {\R}^{100}$. Let the entries of the error vector 
${\be}\in{\R}^{100}$ be normally distributed with zero mean and be normalized so that 
\begin{equation}\label{noiselevel}
\frac{\|{\be}\|}{\|\widehat{\bb}\|}=1\cdot 10^{-3}.
\end{equation}

The program {\sf shaw} determines a vector that, after suitable scaling, solves the error-free system (\ref{linsysexact}). We denote the scaled vector 
by $\widehat{\bx}$. Some computed results are reported in Table~\ref{tabshaw}.

\begin{table}[tbh]
\centering
\caption{Example 5.1 - Parameter estimation and relative errors.}\label{tabshaw}
\begin{tabular}{ccccc}
 \toprule%
 & & \multicolumn{3}{c}{Parameter choice proposed in this paper}\\
 \cmidrule{3-5}
 $q$ & $\ell$ & $\mu$ & $\bar{\phi}_{q+1,\ell}^{1/2}(\mu)-\phi_{q,\ell}^{1/2}(\mu)$ & $\|\widehat{\bx}-\bx_{\ell,\mu}\|_2/
 \|\widehat{\bx}\|_2$\\
 \midrule
 \multirow{5}*{8}  & 1 & $5.58\cdot 10^{3\phantom{-}}$ & $1.26\cdot 10^{-6}$ & $5.25\cdot 10^{-2}$\\
 & 10 & $2.35\cdot 10^{2\phantom{-}}$ & $2.29\cdot 10^{-8}$ & $5.19\cdot 10^{-2}$\\
 & 100 & $2.19\cdot 10^{1\phantom{-}}$ & $1.19\cdot 10^{-8}$ & $5.18\cdot 10^{-2}$\\
 & 1000 & $2.18\cdot 10^{0\phantom{-}}$ & $1.11\cdot 10^{-8}$ & $5.18\cdot 10^{-2}$\\
 & 10000 & $2.17\cdot 10^{-1}$ & $1.10\cdot 10^{-8}$ & $5.18\cdot 10^{-2}$\\
 \bottomrule
\end{tabular}
\end{table}

We report results for the parameter choice strategy \eqref{condgk2} in Table~\ref{tabshaw2} and compare them to results obtained with the parameter choice proposed in the present paper. The latter choice gives improved approximate solutions already for modest $\ell$-values, while the parameter choice rule \eqref{condgk2} requires very many iterations to be competitive with respect to quality.

\begin{table}[tbh]
\centering
\caption{Example 5.1 - Parameter estimation and relative errors.}\label{tabshaw2}
\begin{tabular}{cccc}
 \toprule%
 & & \multicolumn{2}{c}{Parameter choice \eqref{condgk2}}\\ 
 \cmidrule{3-4}
 $q$ & $\ell$ & $\tilde{\mu}$ & $\|\widehat{\bx}-\bx_{\ell,\tilde{\mu}}\|_2/
 \|\widehat{\bx}\|_2$\\ 
 \midrule
 \multirow{5}*{8}  & 1 & $6.51\cdot 10^{2\phantom{-}}$ & $1.05\cdot 10^{-1}$\\ 
 & 10 & $7.10\cdot 10^{1\phantom{-}}$ & $9.11\cdot 10^{-2}$\\ 
 & 100 & $2.79\cdot 10^{0\phantom{-}}$ & $8.95\cdot 10^{-2}$\\ 
 & 1000 & $7.16\cdot 10^{-1}$ & $8.94\cdot 10^{-2}$\\ 
 & 10000 & $9.65\cdot 10^{-1}$ & $5.05\cdot 10^{-2}$\\ 
 \bottomrule
\end{tabular}
\end{table}
The {\sf lsqr} method with tolerance of $1\cdot 10^{-3}$ terminates after $8$ 
iterations and gives a relative error of $5.19\cdot 10^{-2}$. This is very similar to 
results reported in Table \ref{tabshaw}.
Figure~\ref{figrecshaw} displays a reconstruction determined by our method and Figure \ref{figphishaw} (in logarithmic scale) illustrates Theorem~\ref{thm2b}. $~~~~\Box$

\begin{figure*}[tbh]
\centering
\includegraphics[scale=0.35]{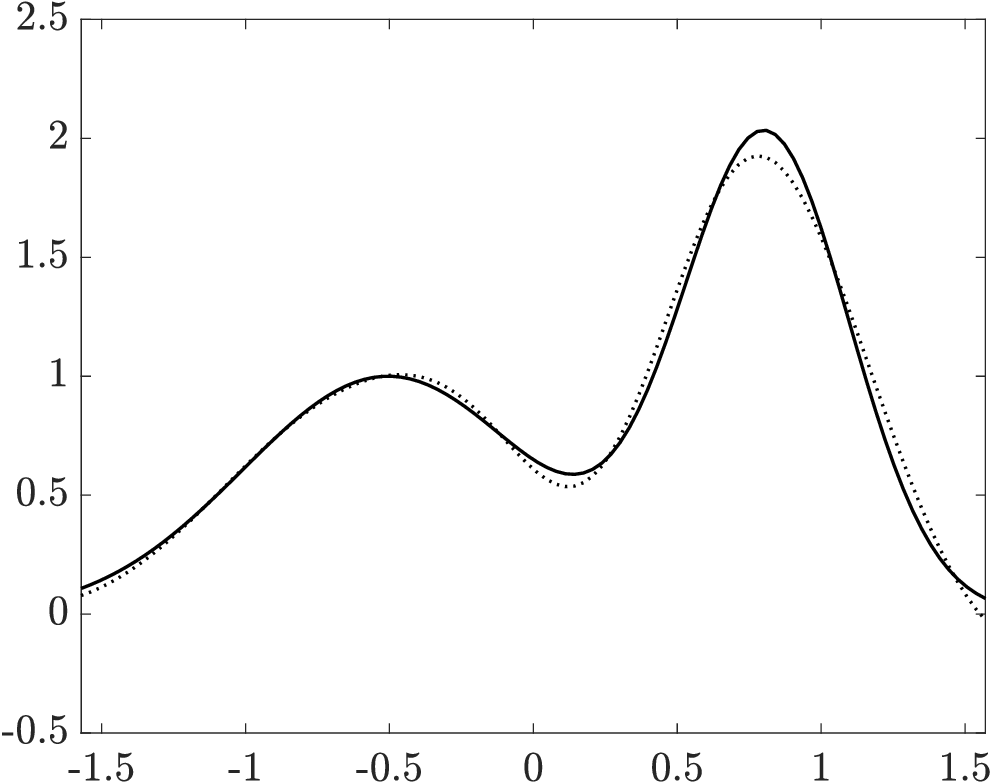} 
\caption{Example 5.1: Solution $\widehat{\bx}$ of the error-free linear system (\ref{linsysexact}) (continuous curve) and computed approximate solution ${\bx}_{\ell,\mu}$ (dotted curve) by the method of the present paper for $q=8$ and $\ell=10$.}\label{figrecshaw}
\end{figure*}

\begin{figure*}[tbh]
\centering
\includegraphics[scale=0.35]{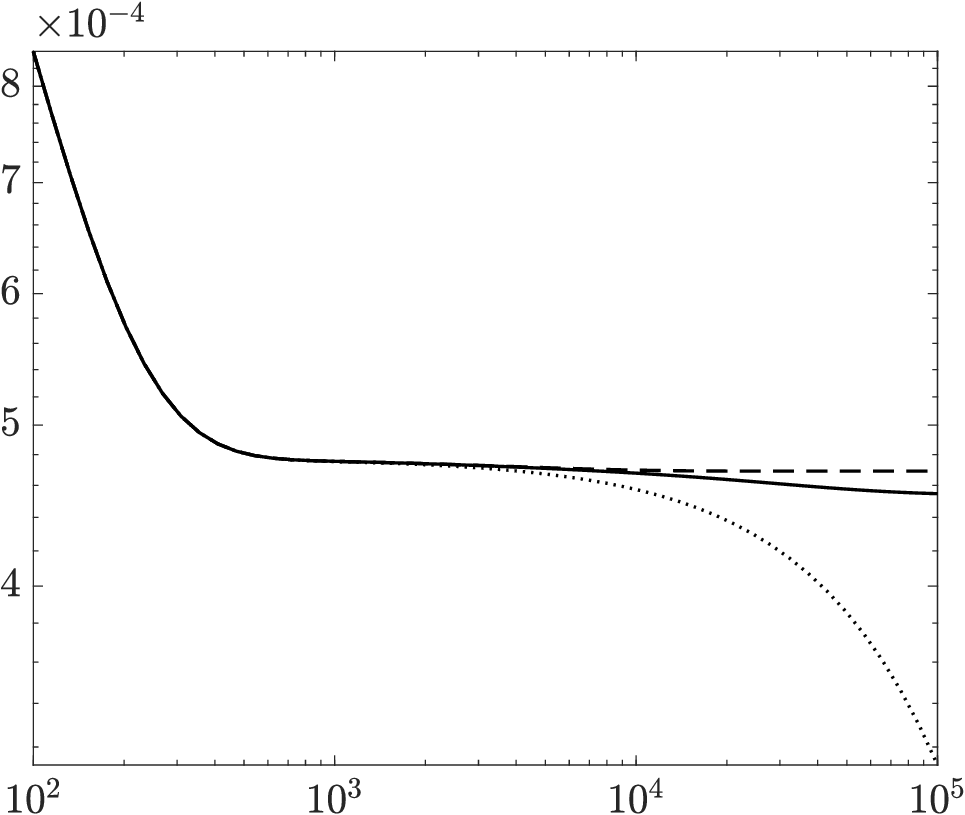}
\caption{Example 5.1 - The function $\phi_{\ell}(\mu)$ and its approximations for $10^2\leq\mu\leq 10^5$. From top to bottom at the right-hand side of the figures: $\bar{\phi}_{q+1,\ell}$ (dashed curve), $\phi_{\ell}$ (continuous curve), $\phi_{q,\ell}$ (dotted curves) for $q=8$ and $\ell=10$.}\label{figphishaw}
\end{figure*}

\vspace{0.5cm}
Example 5.2. The Fredholm integral equation of the first kind,
\begin{equation}\label{baart}
\int_0^{\pi/2} \kappa(\sigma,\tau)x(\sigma)d\sigma=b(\tau),\qquad
0\le \tau\le \pi,
\end{equation}
with $\kappa(\sigma,\tau):=\exp(\sigma\cos(\tau))$, 
$b(\tau):= 2\sinh(\tau)/\tau$, and solution $x(\tau) = \sin(\tau)$ is discussed 
by Baart \cite{Ba}. We use the MATLAB code {\sf baart} from \cite{Ha1} to 
discretize (\ref{baart}) by a Galerkin method with $100$ orthonormal box 
functions as test and trial functions. This yields the matrix 
$A\in{\R}^{100\times 100}$ and the right-hand side vector
$\widehat{\bb}\in {\R}^{100}$. An error vector ${\be}\in{\R}^{100}$ is generated
analogously as in Example 5.1. In particular, 
$\|{\be}\|_2=1\cdot 10^{-3}\|\widehat{\bb}\|_2$. The matrix $A$ and the contaminated
right-hand side ${\bb}$ of (\ref{linsys}) are formed as in Example 5.1; 
they satisfy $\|A^T{\bb}\|_2=2$ and $\|{\bb}\|_2=1$. The matrix $A$ is nonsymmetric 
and has many singular values close to the origin; its condition number is computed to be
$\kappa(A)=1.3\cdot 10^{18}$ with the MATLAB function {\sf cond}. 

Some computed results are displayed in Table~\ref{tabbaart}. We compare the parameter selection strategy of the current paper with the parameter selection strategy using equation \eqref{condgk2}. We can see that the method of this paper determines computed approximations of $\widehat{\bx}$ of higher quality than the method \eqref{condgk2}.

\begin{table}
\centering
\caption{Example 5.2 - Parameter selection and relative errors.}\label{tabbaart}
\begin{tabular}{cccccccc}
 \toprule%
 & & \multicolumn{3}{c}{Method of this paper} & & \multicolumn{2}{c}{Method using \eqref{condgk2}}\\
 \cmidrule{3-5}
 \cmidrule{7-8}
 $q$ & $\ell$ & $\mu$ & $\bar{\phi}_{q+1,\ell}^{1/2}(\mu)-\phi_{q,\ell}^{1/2}(\mu)$ & $\|\widehat{\bx}-\bx_{\ell,\mu}\|_2/
 \|\widehat{\bx}\|_2$ & & $\tilde{\mu}$ & $\|\widehat{\bx}-\bx_{\ell,\tilde{\mu}}\|_2/
 \|\widehat{\bx}\|_2$\\
 \midrule
 \multirow{5}*{5}  & 1 & $4.76\cdot 10^{4}$ & $8.37\cdot 10^{-10}$ & $1.22\cdot 10^{-1}$ & & $7.58\cdot 10^{2\phantom{-}}$ & $1.75\cdot 10^{-1}$\\
 & 10 & $3.26\cdot 10^{3}$ & $9.37\cdot 10^{-11}$ & $1.20\cdot 10^{-1}$ & & $4.97\cdot 10^{1\phantom{-}}$ & $1.68\cdot 10^{-1}$\\
 & 100 & $3.16\cdot 10^{2}$ & $7.03\cdot 10^{-11}$ & $1.20\cdot 10^{-1}$ & & $4.68\cdot 10^{0\phantom{-}}$ & $1.67\cdot 10^{-1}$\\
 & 1000 & $3.15\cdot 10^{1}$ & $6.82\cdot 10^{-11}$ & $1.20\cdot 10^{-1}$ & & $7.02\cdot 10^{-1}$ & $1.64\cdot 10^{-1}$\\
 & 10000 & $3.15\cdot 10^{0}$ & $6.80\cdot 10^{-11}$ & $1.20\cdot 10^{-1}$ & & $9.65\cdot 10^{-1}$ & $1.41\cdot 10^{-1}$\\
 \bottomrule
\end{tabular}
\end{table}

The {\sf lsqr} method with tolerance of $10^{-3}$ terminates after $3$ 
iterations and gives an approximate solution with relative error of $1.66\cdot 10^{-1}$.
This error is about 30\% larger than the error in the approximate solution 
determined by the method of this paper.
Figure~\ref{figrecbaart} displays a reconstruction determined by our method and Figure~\ref{figphibaart} (logarithmic scale) illustrates Theorem~\ref{thm2b}. $~~~\Box$

\begin{figure*}[tbh]
\centering
\includegraphics[scale=0.35]{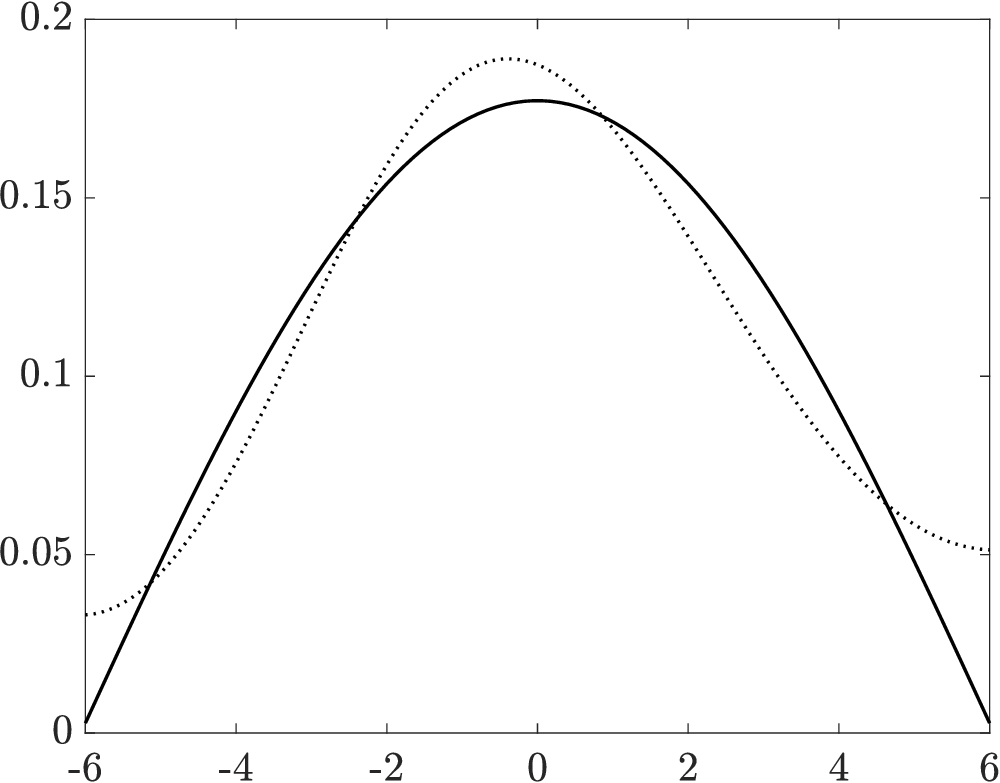} 
\caption{Example 5.2 - Solution $\widehat{\bx}$ of the error-free linear system (\ref{linsysexact}) (continuous curve) and computed approximate solution ${\bx}_{\ell,\mu}$ (dotted curve) by the method of the present paper for $q=5$ and
$\ell=10$.}\label{figrecbaart}
\end{figure*}

\begin{figure*}[tbh]
\centering
\includegraphics[scale=0.35]{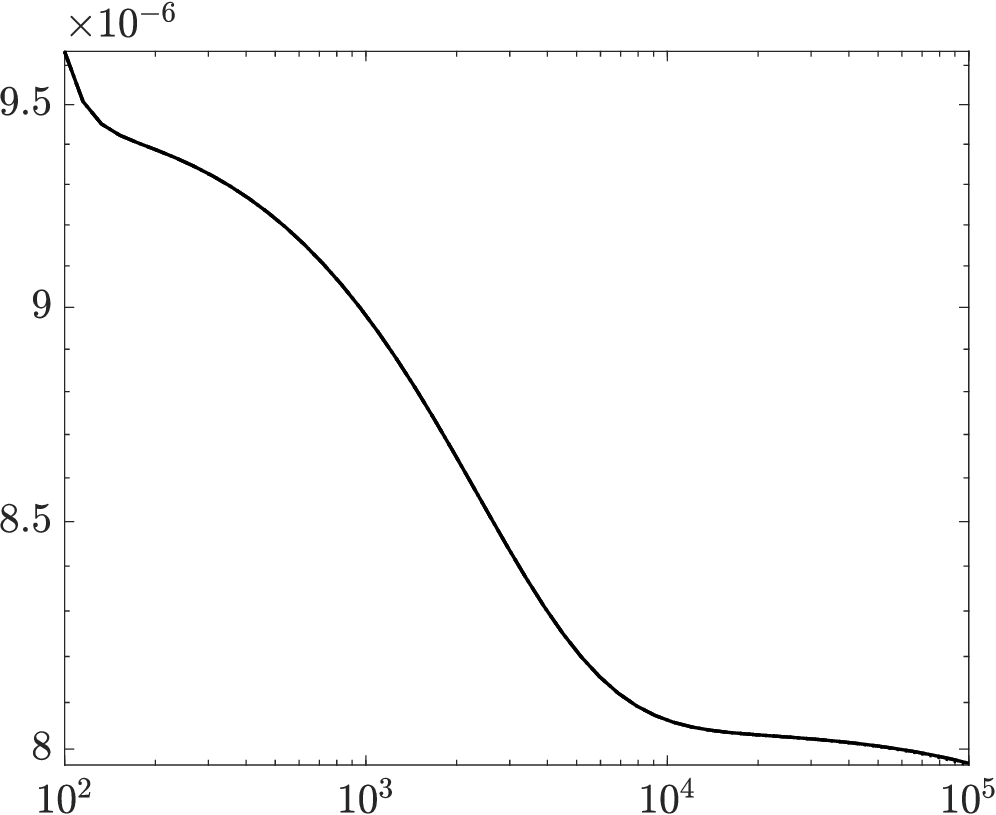}
\caption{Example 5.2 - The function $\phi_{\ell}(\mu)$ and its approximations 
for $10^2\leq\mu\leq 10^5$. From top to bottom at the right-hand side of
the figure: $\bar{\phi}_{q+1,\ell}$ (dashed curve), $\phi_{\ell}$ (continuous curve),
$\phi_{q,\ell}$ (dotted curve) for $q=5$ and $\ell=10$.}\label{figphibaart}
\end{figure*}

\vspace{0.5cm}

Example 5.3. We consider a digital image deblurring problem. We use the function 
\texttt{PRblurshake} from the toolbox \texttt{IRtools} \cite{GHN} with default options and by setting \texttt{seed}=11 and \texttt{generator}=\texttt{`twister'} in the MATLAB function \texttt{rng()} to determine the matrix $A\in{\R}^{512^2\times 512^2}$ that models blurring of an image that is represented by $512\times 512$ pixels and the right-hand side vector
$\widehat{\bb}\in {\R}^{512^2}$. The blur simulates random camera motion (shaking). The true image is represented by the vector $\widehat{\bx}\in {\R}^{512^2}$. An error vector ${\be}\in{\R}^{512^2}$ is generated analogously as in Example 5.1. In particular, $\|{\be}\|_2=4\cdot 10^{-2}\|\widehat{\bb}\|_2$. Some results can be found in Table~\ref{tabblur}.

\begin{table}
\centering
\caption{Example 5.3 - Parameter estimation and relative error.}\label{tabblur}
\begin{tabular}{ccccc}
 \toprule%
 & & \multicolumn{3}{c}{Method of this paper}\\
 \cmidrule{3-5}
 $q$ & $\ell$ & $\mu$ & $\bar{\phi}_{q+1,\ell}^{1/2}(\mu)-\phi_{q,\ell}^{1/2}(\mu)$ & $\|\widehat{\bx}-\bx_{\ell,\mu}\|_2/
 \|\widehat{\bx}\|_2$\\
 \midrule
 \multirow{6}*{50}  & 1 & $1.03\cdot 10^{2\phantom{-}}$ & $7.19\cdot 10^{-7}$ & $2.07\cdot 10^{-1}$\\
 & 5 & $1.04\cdot 10^{1\phantom{-}}$ & $\leq\epsilon$ & $2.03\cdot 10^{-1}$\\
 & 10 & $4.93\cdot 10^{0\phantom{-}}$ & $\leq\epsilon$ & $2.03\cdot 10^{-1}$\\
 & 20 & $2.41\cdot 10^{0\phantom{-}}$ & $\leq\epsilon$ & $2.02\cdot 10^{-1}$\\
 & 50 & $9.50\cdot 10^{-1}$ & $\leq\epsilon$ & $2.02\cdot 10^{-1}$\\
 & 100 & $4.73\cdot 10^{-1}$ & $\leq\epsilon$ & $2.02\cdot 10^{-1}$\\
 \bottomrule
\end{tabular}
\end{table}

Tables~\ref{tabblur} and \ref{tabblur2} compare our approach to determine the regularization parameter with the use of equation \eqref{condgk2}. The entries ``$\leq\epsilon$'' in Table \ref{tabblur} denote that the difference is smaller than machine epsilon. The method described in this paper gives approximations of $\widehat{\bx}$ of higher quality than the method
\eqref{condgk2} already after the first few iterations.
Figure~\ref{figblur} shows a reconstruction determined by our method.

\begin{table}
\centering
\caption{Example 5.3 - Parameter estimation and relative error.}\label{tabblur2}
\begin{tabular}{cccc}
 \toprule%
 & & \multicolumn{2}{c}{Method using \eqref{condgk2}}\\
 \cmidrule{3-4}
 $q$ & $\ell$ & $\tilde{\mu}$ & $\|\widehat{\bx}-\bx_{\ell,\tilde{\mu}}\|_2/
 \|\widehat{\bx}\|_2$\\
 \midrule
 \multirow{5}*{50}  & 1 & $4.05\cdot 10^{1\phantom{-}}$ & $2.37\cdot 10^{-1}$\\
 & 5 & $7.90\cdot 10^{0\phantom{-}}$ & $2.15\cdot 10^{-1}$\\
 & 10 & $3.95\cdot 10^{0\phantom{-}}$ & $2.12\cdot 10^{-1}$\\
 & 50 & $7.90\cdot 10^{0\phantom{-}}$ & $2.10\cdot 10^{-1}$\\
 & 100 & $3.95\cdot 10^{-1}$ & $2.10\cdot 10^{-1}$\\
 \bottomrule
\end{tabular}
\end{table}

The {\sf lsqr} method with tolerance $4\cdot 10^{-2}\|\widehat{\bb}\|$ terminates after 
$10$ iterations and gives a relative error of $1.95\cdot 10^{-1}$. This error is of
about the same size as the errors reported in Table \ref{tabblur}.
Figure~\ref{figphiblur} (logarithmic scale) illustrates Theorem~\ref{thm2b}. $~~~\Box$ 

\begin{figure*}[tbh]
\centering
\includegraphics[scale=0.25]{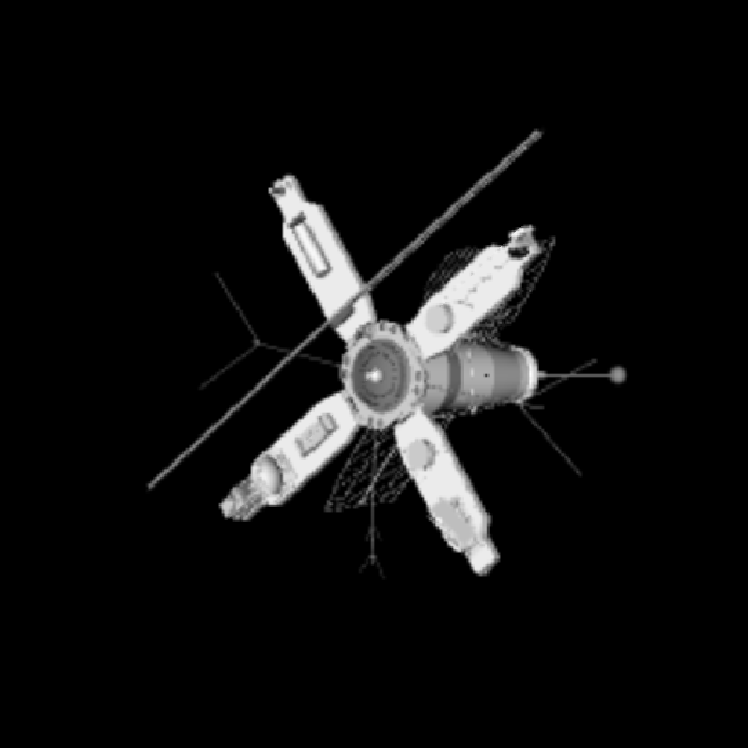}
\qquad
\includegraphics[scale=0.25]{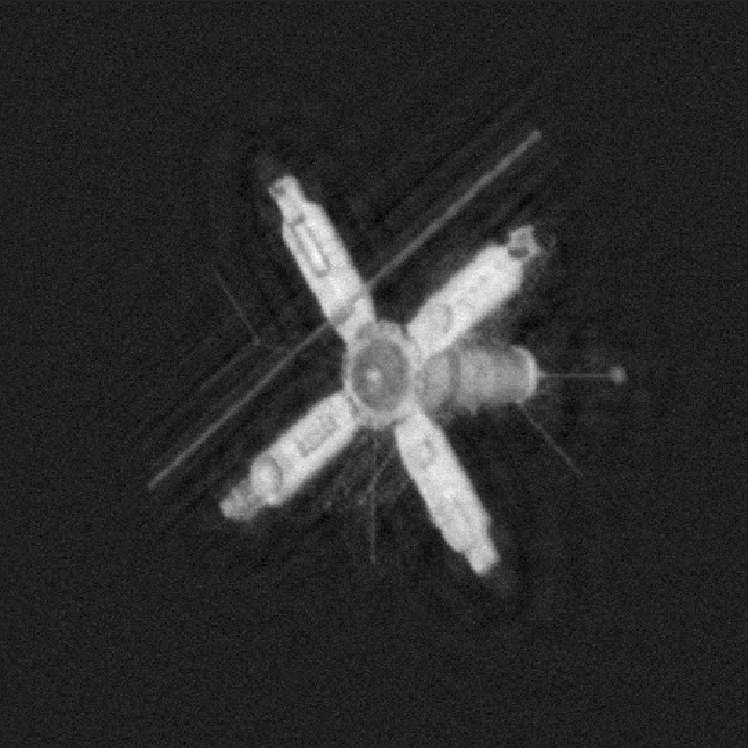}
\qquad
\includegraphics[scale=0.25]{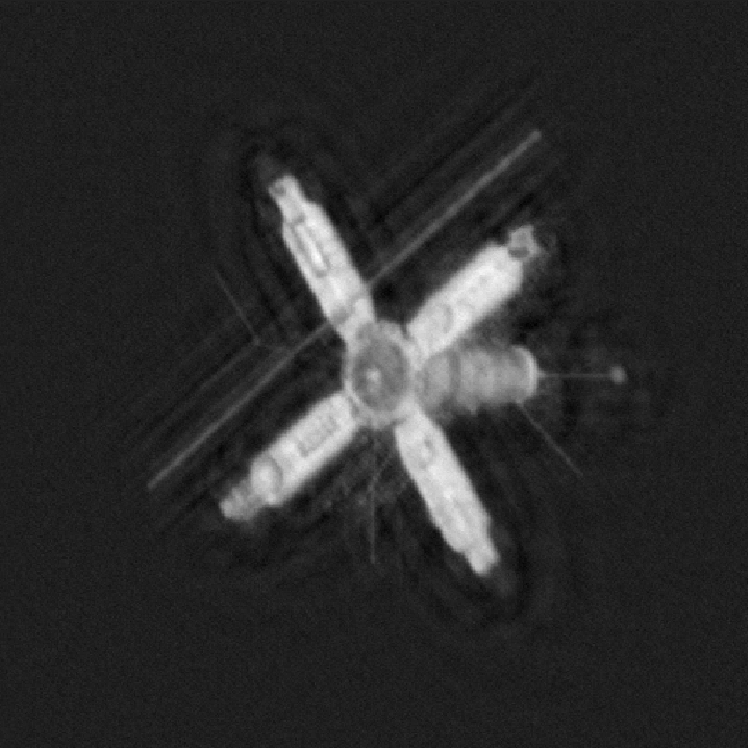} 
\caption{Example 5.3 - Solution $\widehat{\bx}$ (left) of the error-free linear
system (\ref{linsysexact}) and computed approximate solution
${\bx}_{\ell,\mu}$ computed by the method of the present paper for $\ell=1$ (center) and $\ell=20$ (right) for $q=50$.}\label{figblur}
\end{figure*}

\begin{figure*}[tbh]
\centering
\includegraphics[scale=0.35]{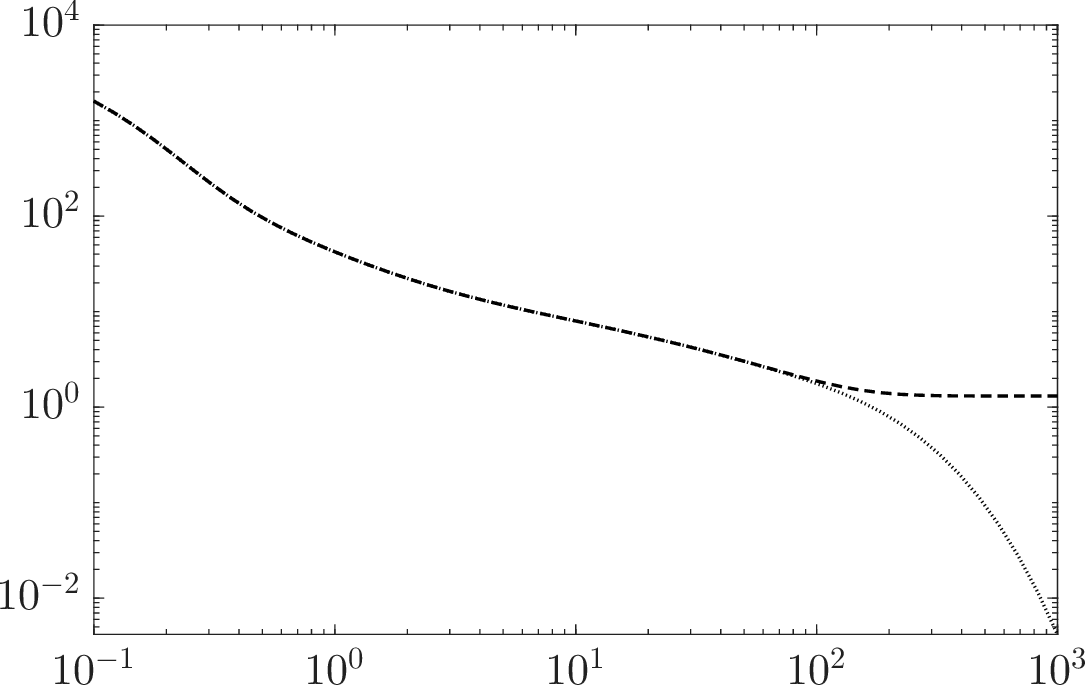}
\caption{Example 5.3 - Approximations of the function $\phi_{\ell}(\mu)$
for $10^2\leq\mu\leq 10^5$. From top to bottom at the right-hand side of the figure: $\bar{\phi}_{q+1,\ell}$ (dashed curve),
$\phi_{q,\ell}$ (dotted curve) for $q=50$ and $\ell=10$.}\label{figphiblur}
\end{figure*}


\section{Conclusion}\label{sec6}
This paper presents a novel approach to choosing the regularization parameter in iterated Tikhonov regularization.
Computed examples show the proposed method to give more
accurate approximate solutions than a recently proposed method discussed in \cite{bdfrGK}. Though, we note that the derivations and analyses of these methods differ; the method discussed in
\cite{bdfrGK} lives in Hilbert space and the choice of the regularization parameter takes the discretization error into account. The method proposed in the present paper is designed to solve linear discrete ill-posed problems and the discretization error is ignored when determining the regularization parameter. 
The method of this paper also is compared lsqr with regularization by
truncated iteration. Our method is found to give the most accurate or close to the
most accurate approximations of the desired solution $\widehat{\bx}$ in all examples.

\printcredits

\bibliographystyle{cas-model2-names}

\bibliography{cas-refs}

\end{document}